\documentclass[11pt,a4paper]{amsart}
\usepackage[utf8]{inputenc}
\usepackage{amsmath}
\usepackage{amsfonts}
\usepackage{amssymb}
\usepackage{amsthm}
\usepackage{setspace}
\usepackage{hyperref}
\usepackage{graphicx}
\usepackage{stmaryrd}
\usepackage[shortlabels]{enumitem}
\usepackage{mathtools}
\usepackage{subcaption}
\usepackage[T1]{fontenc}
\usepackage{xcolor}
\usepackage{float}
\usepackage{tikz-cd}
\usepackage{tikz}
\usepackage[noadjust]{cite}
\usepackage{cite}
\usepackage{ragged2e}
\usepackage[margin=1.2in]{geometry}
\graphicspath{ {./images/} }
\newtheorem{theorem}{Theorem}[section]

\newtheorem{lemma}[theorem]{Lemma}
\newtheorem{corollary}[theorem]{Corollary}

\theoremstyle{definition}
\newtheorem{definition}[theorem]{Definition}
\newtheorem{example}[theorem]{Example}

\newtheorem{question}[theorem]{Question}
\newtheorem{case}{Case}
\newtheorem{mycase}{Case}
\newtheorem{claim}{Claim}


\title[Boundary rigidity of systolic and Helly complexes] {Boundary rigidity of systolic and Helly complexes}
\author[M.~Blufstein]{Martín Blufstein}
\address{Department of Mathematical Sciences, University of Copenhagen, 2100 Copenhagen, Denmark}
\email{mblufstein@dm.uba.ar}
\author[J.\ Chalopin]{J\' er\'emie Chalopin}
\address{CNRS and Aix-Marseille Universit\'e, LIS, Marseille, France
}
\email{jeremie.chalopin@lis-lab.fr}

\author[V.\ Chepoi]{Victor Chepoi}
\address{Aix-Marseille Universit\'e and CNRS, LIS, Marseille, France}
\email{victor.chepoi@lis-lab.fr}

\date{}

\begin{document}

\begin{abstract}
\centering \justifying In this article, we prove that finite (weakly) systolic  and Helly complexes can be reconstructed from their boundary distances (computed in their 1-skeleta). Furthermore, Helly complexes and 2-dimensional systolic complexes can be reconstructed by an algorithm that runs in polynomial time with respect to the number of vertices of the complex. 
Both results can be viewed as a positive contribution to a general question of  Haslegrave,
Scott, Tamitegama, and Tan (2025).  The reconstruction of a finite cell complex from the boundary distances is the discrete analogue of the boundary rigidity problem, which is a classical problem from Riemannian geometry.  
\end{abstract}

\maketitle

\section{Introduction}

A natural question, common to several research areas, is whether the
internal structure of an object can be determined from the distances
between the points in its boundary. A first positive answer to such question comes from a classical result in phylogeny by
Buneman~\cite{Bu} and Zareckii~\cite{Za}, which states that trees can be reconstructed
from the pairwise distances between their leaves.
This question also shows up in Riemannian
geometry, where the notion of reconstruction from a distance function on the
boundary is well-established.
A Riemannian manifold $(M,g)$ is said to be \emph{boundary rigid} if its metric $d_g$ is determined up to isometry by its boundary distance function.
Michel~\cite{Mi} conjectured that all simple
compact Riemannian manifolds with boundary are boundary rigid. The case
of 2-dimensional Riemannian manifolds was confirmed by Pestov and
Uhlmann~\cite{PeUh}. In higher dimensions, the conjecture is still open and 
has been confirmed in only two cases by Besson, Courtois and
Gallot~\cite{BeCoGa}, and by Burago and Ivanov~\cite{BuIv}. 
A discrete version of
boundary rigidity was proposed by Benjamini (see~\cite{Ha,HaScTaTa})
who asked if any plane triangulation 
in which all inner vertices have degrees at least 6 can be
reconstructed from the distances between its boundary vertices. This
was answered in the affirmative by Haslegrave~\cite{Ha}, who also
proved a similar result for plane
quadrangulations 
in which all inner vertices have degree at least 4. Haslegrave,
Scott, Tamitegama, and Tan~\cite{HaScTaTa} generalized the second
result of~\cite{Ha} and proved that any finite 2-dimensional CAT(0)
cube complex and any finite 3-dimensional CAT(0) cube complex $X$ with
an embedding in ${\mathbb R}^3$  can be
reconstructed from its boundary distances. They conjectured (Conjecture 4 of~\cite{HaScTaTa}) that all finite CAT(0) cube complexes are boundary rigid and they asked which simplicial complexes are boundary rigid (Question 2 of~\cite{HaScTaTa}).  
The second and third authors resolved in~\cite{ChCh_bd} Conjecture 4 and showed that all finite CAT(0) cube complexes can be reconstructed up to isomorphism from their boundary distances. In this article, we contribute to the general Question 2 of~\cite{HaScTaTa} and show that two important classes of simplicial complexes  with combinatorial nonpositive curvature (namely, bridged and Helly complexes) are boundary rigid. The boundary rigidity of these  classes of simplicial complexes was asked in~\cite{ChCh_bd}.  

Bridged graphs constitute an important class of graphs in metric graph theory \cite{BaCh_survey}.
They have been introduced in the papers \cite{SoCh} and \cite{FaJa} where it is proven that bridged graphs are exactly the graphs which satisfy the following geometric property (which is one of the basic properties of Euclidean convexity and CAT(0) spaces): balls around convex sets are convex.
Bridged graphs have been characterized
in \cite{SoCh,FaJa} as the graphs in which all isometric cycles have length 3. Chordal graphs are examples of bridged graphs. 
It was shown in \cite{AnFa} that bridged graphs are dismantlable. Furthermore, any BFS-order of a bridged graph is a dismantling order~\cite{Ch_bridged}; see \cite{Po1} for dismantlability of non-locally finite bridged graphs.

The third author \cite{Ch_CAT} presented a local-to-global characterization of bridged graphs as the 1-skeleta of simply connected flag simplicial complexes with 6-large links
(i.e., they do not have induced 4- and 5-cycles).
Those simplicial complexes have been rediscovered by Januszkiewicz and \'{S}wia\c{t}kowski \cite{JaSw} (and independently by Haglund \cite{Haglund}) and dubbed systolic complexes.
In dimension 2, systolic complexes are exactly the CAT(0) simplicial complexes where every triangle is endowed with the metric of a Euclidean equilateral triangle of side length 1 (this is not longer true for systolic complexes of dimension 3 or more \cite{Ch_CAT,JaSw}). 
Systolic complexes are contractible \cite{Ch_CAT,JaSw} and they are considered as simplicial analogues of
CAT(0) cube complexes.
Groups acting on systolic complexes/bridged graphs and systolic complexes arising from groups have been studied in numerous papers;
\cite{HoOs,HuOs_systolic,JaSw,OsPr,Pr} is a small sample. 
Weakly bridged graphs \cite{ChOs} generalize bridged graphs and are exactly the weakly modular graphs with convex balls. 
The 2-dimensional systolic complexes (alias $K_4$-free bridged graphs) have been investigated in the papers \cite{ChLaRa} and \cite{HoOs}.
Both papers present solutions in the 2-dimensional case for questions which remain open for all systolic complexes. To 
understand how general 2-dimensional systolic complexes are, for any finite graph $H$ of girth $g\ge 6$, Weetman \cite{We} presented a nice construction $W(H)$ of (infinite) graphs in which the links (neighborhoods) of all
vertices are isomorphic to  $H$.  For example, if $H = C_6$, then $W(H)$ is the triangular grid. If $H = C_n$ for $n \ge 7$, then $W(H)$ is the graph of a tiling
of the (hyperbolic) plane into triangles in which each vertex has $n$
neighbors. For general $H$, $W(H)$ is always bridged 
and $K_4$-free, thus its clique complex is a 2-dimensional systolic complex. Any finite simply connected subcomplex of $W(H)$ is systolic and 2-dimensional. 

We show that finite bridged graphs/systolic complexes are boundary rigid.

\begin{theorem}\label{thm:boundary_rigidity_systolic}
    Let $X$ be a finite systolic complex. Then $X$ is determined up to isomorphism by the combinatorial distances of vertices of $\partial X$ in $X$. Hence, the class of finite systolic complexes is boundary rigid. 
\end{theorem}

Theorem \ref{thm:boundary_rigidity_systolic} generalizes in a straightforward way to weakly systolic complexes (see Corollary \ref{cor:weakly}). 
Our proof is based on the fact that among the furthest vertices of a bridged graph are dominated vertices (which follows from \cite{Ch_bridged}), and presents an algorithm to determine $X$ from the distance matrix of $\partial X$.
However, this reconstruction algorithm is not polynomial because we cannot find such dominated vertices, based only on the boundary information.
In the case of 2-dimensional systolic complexes we can find such dominated vertices efficiently, and the reconstruction can be done in polynomial time.  

\begin{theorem}\label{thm:boundary_rigidity_systolic-2dim}
  Let $X$ be a finite 2-dimensional systolic complex. Then $X$ is
  determined up to isomorphism by the combinatorial distances of
  vertices of $\partial X$ in $X$. Furthermore, $X$ can be reconstructed
  in time polynomial in the size of $X$.
\end{theorem} 

Helly graphs are the graphs in which the balls satisfy the Helly property, i.e., any collection
of pairwise intersecting balls has a nonempty intersection. We call the clique complexes of Helly graphs \emph{Helly complexes}.
Helly graphs are the discrete analogues of hyperconvex spaces \cite{ArPa}  (complete geodesic metric
spaces in which the set of closed balls satisfies the Helly property).
In perfect analogy with hyperconvexity, there is a close relationship between Helly graphs and absolute retracts.
A graph is an absolute retract exactly when it is a retract of any larger graph into which it embeds isometrically. Then absolute retracts and Helly graphs are the same.
Helly graphs have been characterized metrically  in the papers \cite{BaPe,BaPr} and topologically in \cite{CCHO}, in a local-to-global way. The study of groups acting
on Helly graphs was initiated in the paper \cite{CCGHO} and continued in many other papers, e.g. \cite{HuOs_Helly,OsVa}.
Numerous examples of such groups are also
known: hyperbolic groups and groups acting on CAT(0) cube complexes are Helly \cite{CCGHO}.

Our third result shows that finite Helly complexes are boundary rigid and can be reconstructed in output-polynomial time: 

\begin{theorem}\label{thm:boundary_rigidity_helly}
  Let $X$ be a finite Helly complex. Then $X$ is determined up to
  isomorphism by the combinatorial distances of vertices of
  $\partial X$ in $X$.  Hence, the class of finite Helly complexes is
  boundary rigid.  Furthermore, $X$ can be reconstructed in time
  polynomial in the size of $X$.
\end{theorem} 

\subsection*{Acknowledgements} 
The work presented here was partially supported by the Carlsberg Foundation, grant CF23-1225. J.C. and V.C. were also partially supported by the ANR project MIMETIQUE ``Mineurs métriques'' and J.C. was partially supported by the ANR project DUCAT (ANR-20-CE48-0006). 

\section{Preliminaries}

\subsection{Graphs and simplicial complexes}
All graphs $G=(V(G),E(G))$ considered in this paper are finite, undirected,
connected, and contain neither multiple edges, nor loops. For two distinct
vertices $v,w\in V(G)$ we write $v\sim w$ when there is an edge
connecting $v$ with
$w$.  The subgraph
of $G$ \emph{induced by} a subset $A\subseteq V(G)$ is the graph
$G[A]=(A,E')$ such that $uv\in E'$ if and only if $uv\in E(G)$. 
The \emph{distance}
$d(u,v)=d_G(u,v)$ between two vertices $u$ and $v$ of a graph $G$ is the
length of a shortest $(u,v)$--path. 
An induced subgraph $H=G[A]$ of a graph $G$ is an \emph{isometric subgraph} of $G$ if $d_H(u,v)=d_G(u,v)$ for any
two vertices $u,v\in A$. The \emph{interval}
$I(u,v)$ between $u$ and $v$ consists of all vertices on shortest
$(u,v)$--paths, that is, of all vertices (metrically) \emph{between} $u$
and $v$: $I(u,v)=\{ x\in V(G): d(u,x)+d(x,v)=d(u,v)\}$. An induced
subgraph of $G$ 
is called \emph{convex}
if it includes the interval of $G$ between any pair of its
vertices.  For a vertex $v$ of $G$ and an integer $r\ge 1$, we will denote  by $B_r(v)$ the \emph{ball} in $G$
(and the subgraph induced by this ball)  of radius $r$ centered at  $v$, i.e., $B_r(v)=\{ x\in V(G): d(v,x)\le r\}.$ More generally, the $r$--{\it ball  around a set} $A\subseteq V(G)$
is the set (or the subgraph induced by) $B_r(A)=\{ v\in V(G): d(v,A)\le r\},$ where  $d(v,A)=\mbox{min} \{ d(v,x): x\in A\}$.
Similarly, the \emph{sphere} in $G$ centered at $v$ of radius $r$ is the set (or the subgraph induced by) $S_r(v)=\{ x \in V(G): d(v,x) = r\}$.
As usual, $N(v)=S_1(v)$ denotes the set of neighbors of a vertex
$v$ in $G$, and $N[v]=B_1(v)$.
The \emph{link} of $v \in V(G)$ is the subgraph of $G$ induced by $N(v)$. 

An \emph{abstract simplicial complex} $X$ on a finite set
$V$ is a set of nonempty subsets of $V$, called \emph{simplices}, such
that any nonempty subset of a simplex is also a simplex.  The dimension of a simplex $\sigma$ is $|\sigma|-1$ and the dimension of $X$ is the largest dimension of
a simplex of $X$.  If $\sigma$ is a simplex of dimension $k$, then all
simplices of $X$ contained in $\sigma$ and having dimension $k-1$ are called
the \emph{facets} of $\sigma$. Simplices of $X$ are called \emph{maximal} if they are maximal by inclusion, and \emph{non-maximal} otherwise. 
The 2-dimensional simplices of $X$ are called \emph{triangles}, and if a triangle is spanned by vertices $u$, $v$ and $w$, then we denote it by $uvw$.

For an abstract simplicial complex $X$, the \emph{$k$--skeleton}  $X^{(k)}$  of $X$ consists of all simplices of dimension at most $k$.
We use the notations $V(X) = X^{(0)}$ for the set of vertices
(0-simplices) of $X$ and $G(X):=X^{(1)}$ for the 1-skeleton of $X$; the
graph $G(X)$ will be always endowed with the standard graph distance
$d_G$.
The \emph{clique complex} of a graph $G$ is the abstract simplicial complex having the cliques (i.e., complete subgraphs) of $G$ as simplices.
A simplicial complex $X$ is a \emph{flag simplicial complex} if $X$ is the clique complex of its $1$--skeleton.
Note that flag complexes are uniquely determined by their 1-skeleton.

Each abstract  simplicial complex $X$ gives rise to a geometric simplicial complex $||X||$ obtained by replacing each simplex $\sigma$ of $X$ by an Euclidean simplex $||{\sigma}||$ such that  gluing maps between simplices are isometries (see the book of Hatcher~\cite{Hat}).
Vice-versa, the set of simplices of a geometric simplicial complex define an abstract simplicial complex.
Although the definitions of boundary and boundary rigidity presented in the next subsection are motivated by geometric simplicial complexes, all these notions can be defined in equivalent way for abstract simplicial complexes.
Therefore, in the rest of the paper by a ``simplicial complex'' we will mean a \emph{finite abstract simplicial complex}. 

\subsection{Boundary rigidity of simplicial complexes} The notions of combinatorial boundary and boundary rigidity were defined by Haslegrave et al.~\cite{HaScTaTa} for all cell complexes.
We will use a slightly weaker version of the definition of combinatorial boundary and we will define it only for (abstract) simplicial complexes. 
For a  simplicial complex $X$, the \emph{combinatorial boundary}
$\partial X$ is the downward closure of all non-maximal simplices of $X$ such that each of them is a facet of a unique simplex of $X$.
The definition of~\cite{HaScTaTa} is similar but instead of maximality by inclusion the maximality by dimension is considered:
in~\cite{HaScTaTa}, the combinatorial boundary of a simplicial complex $X$ of dimension $k$ is the downward closure of all simplices of $X$ of dimension less than $k$ belonging to at most one cell of dimension $k$.
Obviously, the boundary of a simplicial complex $X$ defined in our way is always contained in the boundary of $X$ defined as
in~\cite{HaScTaTa}.
The two boundaries are different as soon as the simplicial complex $X$ contains maximal cells of different dimensions; see also \cite{ChCh_bd} for the case of CAT(0) cube complexes.
A \emph{boundary distance matrix} of a simplicial complex $X$ is a matrix $D(X)$ whose rows and columns are the vertices of $\partial X$ and whose entries are the distances in $G(X)$ between the corresponding vertices of $\partial X$.
\begin{definition}[Boundary rigidity]
A class $\mathfrak{X}$ of finite simplicial
complexes is called \emph{boundary rigid} if for any two complexes $X,Y \in \mathfrak{X}$ that have the same boundary distance matrices up to a permutation $\pi$ of the rows and the columns ($\pi$ can be seen as an isometry between $\partial X$ and $\partial Y$), $X$ and $Y$ are isomorphic via an isomorphism $\widetilde{\pi}$ extending $\pi$.
\end{definition}

\subsection{Bridged and Helly graphs} 

A graph $G$ is \emph{weakly modular} 
if $G$ satisfies the triangle and quadrangle conditions:
\begin{enumerate}
    \item[(TC)] \emph{Triangle Condition}: for  $v,x,y\in V$ with $d(v,x)=d(v,y)=k$ and $x\sim y$, $\exists z\in B_{k-1}(v)$, $z\sim x,y$. 
    \item[(QC)] \emph{Quadrangle Condition}: for $v,x,y,u\in V$ with $d(v,x)=d(v,y)=k=d(v,u)-1$ and $u\sim x,y$, $x\nsim y$, $\exists z\in B_{k-1}(v)$, $z\sim x,y$.
\end{enumerate}

A graph $G$ is \emph{bridged} if any isometric cycle of $G$ has size 3. By \cite{FaJa,SoCh}, bridged graphs are exactly the
graphs in which the balls $B_r(A)$ around convex sets $A$ are convex. Bridged graphs are weakly modular, in fact they are exactly the weakly modular graphs not containing induced $C_4$ and $C_5$ \cite{Ch_CAT}. 
The clique complexes of bridged graphs are called \emph{systolic complexes}. By the results of \cite{Ch_CAT,JaSw}, a simplicial complex $X$ is systolic  if and only if $X$ is flag, simply connected and 6-large (which means that the links of vertices of $G(X)$ do not contain induced 4-cycles and 5-cycles).  The \emph{weakly bridged graphs} are the weakly modular graphs with convex balls. By \cite{ChOs} they are exactly the weakly modular graphs not containing induced $C_4$. The clique complexes of weakly bridged graphs were called \emph{weakly systolic complexes} and, similarly to systolic complexes, have been characterized in a local-to-global way in \cite{ChOs}. \emph{Helly graphs} are the graphs in which the balls satisfy the Helly property, i.e., any collection
of pairwise intersecting balls has a nonempty intersection. Helly graphs are weakly modular. We call clique complexes of Helly graphs \emph{Helly complexes}. The Helly complexes have been characterized in \cite{CCHO} as the flag, simply connected clique-Helly simplicial complexes (clique-Hellyness means that the set of maximal simplices satisfies the Helly property). 

We continue with some properties of bridged and Helly graphs used in this paper. We start with two simple known results. 

\begin{lemma}[{\cite[Proposition 4.3]{ChOs}}]\label{l:clique-in-the-sphere} Let $G$ be a bridged or weakly bridged graph, $x$ be any vertex of $G$, and $k$ a positive integer. If $K$ is a clique of the sphere $S_k(x)$, then the set of all vertices $y\in S_{k-1}(x)$ adjacent to all vertices of $K$ defines a nonempty clique of $G$.   
\end{lemma}

\begin{lemma}[{\cite[Theorem 1(iv)]{BaPe}}]\label{l:Helly-BP} Let $G$ be a Helly graph, and $z$ and $u$ be two vertices of $G$ with $k=d(z,u)\ge 2$. Then there exists a vertex $v$ which has distance $k-1$ to $z$ and is adjacent to $u$ and to all neighbors $w$ of $u$ with $d(z,w)\le k$. 
\end{lemma}

A consequence of Lemma \ref{l:clique-in-the-sphere} is the following.

\begin{lemma}\label{l:sphere-in-boundary}
    Let $X$ be a systolic or weakly systolic complex. Let $z \in X$ and $k = \max_{x \in V}\{d(x,z)\}$. Then $S_k(z) \subseteq \partial X$.
\end{lemma}

\begin{proof}
    A given vertex $x$ of the sphere $S_{k}(z)$ belongs to a maximal clique $K$ of $S_{k}(z)$. From the definition of $k$ and Lemma \ref{l:clique-in-the-sphere}, $K$ is included in a unique maximal clique $K'$ of $G$. Consequently, all vertices of $K$, and in particular the vertex $x$, belong to $\partial X$.
\end{proof}

A vertex $u$ of a graph $G$ is \emph{dominated} by another
vertex $v$ if the unit ball $B_1(v)$  includes $B_1(u)$. A finite graph $G=(V,E)$ is \emph{dismantlable} if its
vertices can be ordered $v_1,\ldots,v_n$ so that, for each $v_i$ there is a neighbor $v_j$ of $v_i$ with
$j<i$ which dominates $v_i$ in the subgraph $G_i$ of $G$ induced by the vertices $v_1,\ldots,v_i$. 

\begin{lemma}\label{lem:dominated_neighborhood}
    Let $X$ be a simplicial complex, and $u$ and $v$ vertices of $X$ such that $v$ dominates $u$ in $G(X)$.
    Then $N[u]\setminus \{ v\}\subseteq \partial X$. Consequently, if $u$ is dominated by two distinct vertices $v,v'$, then $N[u] \subseteq \partial X$. 
\end{lemma}

\begin{proof}
    Let $w \in N(u)\setminus\{v\}$, and let $K$ be a maximal clique containing $u$ and $w$.
    Since $v$ dominates $u$, then $v \in K$. Furthermore, $K'=K\setminus \{v\}$ is a facet of $K$. If $w\notin \partial X$,
    then there exists a maximal clique $K''$ containing $K'$ and different from $K$. Then there exists $x\in K''\setminus K'$ such that $x\sim u$, but $x\not\sim v$ (since $K''\neq K$), a contradiction.
    Hence $w \in \partial X$. The second assertion follows by applying the first assertion for $v$ and $v'$.
\end{proof}

Dismantlabilitly of bridged graphs was first proved in \cite{AnFa}. But we need more specific dismantling orders of bridged and weakly bridged graphs, which can be useful in our setting. They are provided by Breadth-First-Search (BFS) and Lexicographic Breadth-First-Search (LexBFS) algorithms:

\begin{theorem}\label{dismantling-bridged}\ 

\begin{enumerate}[(1)]
\item\cite{Ch_bridged} If $G=(V,E)$ is a bridged graph, then any ordering $v_1,\ldots,v_n$ of vertices of $G$ returned by BFS is a dismantling order of $G$. 

\item\cite{ChOs} If $G=(V,E)$ is a weakly bridged graph, then any ordering $v_1,\ldots,v_n$ of vertices of $G$ returned by LexBFS is a dismantling order of $G$. 
\end{enumerate}
\end{theorem}

Since BFS and LexBFS order the vertices of $G$ by distances with respect to the basepoint $z=v_1$, 
the last vertex $v_n$ of the order $v_1,\ldots,v_n$ is a dominated vertex of $G$ and is a furthest from $z$ vertex. Consequently, we have the following corollary of Theorem \ref{dismantling-bridged}: 

\begin{lemma}\label{lem:dominated_far_away}
    Let $G=(V,E)$ be a bridged or weakly bridged graph, $z\in V$, and 
    $k = \max_{x \in V}\{d(x,z)\}$.
    Then $S_k(z)$ contains a dominated vertex $u$. Furthermore, 
    the subgraph $G'$ of $G$ induced by $V\setminus \{ u\}$ is a bridged graph, which is an isometric subgraph of $G$. 
\end{lemma}

The second assertion of Lemma \ref{lem:dominated_far_away} follows from the fact that $G'$ is an isometric subgraph of $G$ (since $u$ is dominated) and any isometric subgraph of a bridged graph is also bridged. 

That Helly graphs are dismantlable was proved in \cite{BaPe}. The
proof of \cite{BaPe} is based on Lemma~\ref{l:Helly-BP} and the
following lemma is an analogue of Lemma \ref{lem:dominated_far_away}
for Helly graphs:

\begin{lemma}[\cite{BaPe}] \label{lem:dominated_far_away_Helly} Let
  $G=(V,E)$ be a Helly graph and $z\in V$.  Let
  $k \coloneq \max_{x \in V}\{d(x,z)\}$.  Then any vertex
  $u\in S_k(z)$ is a dominated vertex. Furthermore, the subgraph $G'$
  of $G$ induced by $V\setminus \{ u\}$ is an isometric subgraph of
  $G$ and $G'$ is a Helly graph. 
\end{lemma} 

The first assertion of Lemma \ref{lem:dominated_far_away_Helly} follows by applying Lemma \ref{l:Helly-BP} to $z$ and any vertex $u$ furthest from $z$. The second assertion of Lemma \ref{lem:dominated_far_away_Helly} follows from the simple observation that since $v$ dominates $u$, $d(v,x)\le d(u,x)$, thus  any set of balls intersecting at $u$ also contains $v$ in the intersection.

\section{Boundary rigidity of systolic complexes}

The goal of this section is to prove Theorem
\ref{thm:boundary_rigidity_systolic}. 
 Let $X$ be a finite systolic
complex and $G=(V,E)$ be a bridged graph, which is the 1-skeleton of
$X$. Let $\partial X$ be the boundary of $X$ and $D$ be the distance
matrix of pairwise distances in $G$ between the vertices of
$\partial X$. Let $z$ be an arbitrary fixed vertex of $\partial X$. We
proceed by induction on the number of vertices of $X$. In fact, we
will prove a slightly more general inductive hypothesis: \emph{if
  $X,Y$ are two finite systolic complexes and $S(X)$ is a subset of
  vertices of $X$ such that $\partial X\subseteq S(X),$
  $\partial Y\subseteq S(X)\subseteq V(Y)$, and $D$ is the distance
  matrix of $S(X)$ in $G(X)$ and $G(Y)$, then $X=Y$.}

If $X$ contains a single vertex, then both $X$ and $\partial X$ consist only of a single 0-simplex $\{ z\}$. Therefore, suppose that $X$ contains at least two vertices. Applying Lemma \ref{lem:dominated_far_away} to $z\in \partial X$, we conclude that $G=G(X)$ contains a dominated vertex $u$ different from $z$. By Lemma \ref{lem:dominated_neighborhood}, $u\in \partial X$. Suppose that $u$ is dominated in $G$ by a vertex $v$. Let $X'$ be the simplicial complex consisting of all simplices 
of $X$ not containing the vertex $u$ and let $G'$ be the subgraph of $G$ induced by $V\setminus \{ u\}$. Then $G'$ is the 1-skeleton of $X'$. Since $u$ is dominated by $v$, $G'$ is an isometric subgraph of $G$. Since $G$ is bridged, $G'$ is also bridged. Thus $X'$ is the clique complex of 
$G'$, whence $X'$ is systolic. 

Next, from $S(X)$ and the distance matrix $D$ of $S(X)$ in $G$, we define a set $S(X')$ such that $\partial X'\subseteq S(X')\subseteq V'=V\setminus \{ u\}$ and we show how to compute the distance matrix of $S(X')$ in $G$ (and thus in $G'$). We distinguish two cases.

\begin{case} There exists a vertex $v'\in S(X), v'\sim u$ such that $N[u]\cap S(X)\subseteq N[v']$.
\end{case}

We assert that in this case $v'$ dominates $u$ in $G$, i.e., that $N[u]\subseteq N[v']$. If $v'=v$, then we are done because $v$ dominates $u$ in $G$. Now suppose that $v'\ne v$. Pick any $w\in N[u]$. If $w=v$, then $w\sim v'$ because $v$ dominates $u$ in $G$. If $w\ne v$, then $w\in \partial X\subseteq S(X)$ by Lemma \ref{lem:dominated_neighborhood}. Since $N[u]\cap S(X)\subseteq N[v']$, $w\sim v'$ also in this case, and we are done. In this second case, $u$ is dominated by $v$ and $v'$. Consequently, by the second assertion of Lemma \ref{lem:dominated_neighborhood}, $v\in \partial X\subseteq S(X)$. Therefore, in both cases ($v'=v$ or $v'\ne v$) we can set $S(X')=S(X)\setminus \{ u\}$. The distance matrix $D'$ of $S(X')$ is obtained from the distance matrix $D$ of $S(X)$ by removing the row and the column of $u$.

\begin{case} There does not exist a vertex $v'\in S(X), v'\sim u$ such that $N[u]\cap S(X)\subseteq N[v']$.
\end{case}

In this case, we set $X' = X\setminus\{u\}$, and $S(X') = S(X)\setminus\{u\}\cup\{v\}$. To derive the distance matrix $D'$ of $S(X')$ from the distance matrix $D$ of $S(X)$, we need to compute the distances $d(v,x)$ from $v$ to all vertices $x$ of $S(X')$. Pick any $x\in S(X')$. If $x\in N(u)$, then $d(v,x) = 1$ since $v$ dominates $u$.
Otherwise, let $k=d(x,u)$. If there exists $w \in S(X')\cap N(u)$ with $d(x,w) = k+1$, then we can set $d(x,v) = k$. Indeed, $u$ belongs to a shortest $(w,x)$-path $P$ of $G$. If $y$ is the neighbor of $u$ in $P$, then $v\sim w,y$  because $v$ dominates $u$ in $G$. Consequently, $v$ is on a shortest $(w,x)$-path, whence $d(v,x)=k$. Notice also that since 
$w \in S(X')\cap N(u)$, $w$ also belongs to $S(X)$. Thus, the entry $d(v,x)$ of $D'$ is derived from the entry $d(w,x)$ of $D$. 

Now suppose that $d(w,x)\le k$ for every vertex $w\in S(X')\cap N(u)$. 
We assert that in this case $d(v,x)=k-1$. Indeed, suppose it is not the case. Then necessarily, we must have  $d(x,v) = k$.
Let $K$ be a maximal simplex of the sphere $S_k(x)$ containing $u$ and $v$.
Consider a maximal simplex $K'$ of $X$ containing $K$. Both $K$ and $K'$ are cliques of $G$.  Since $u$ is a local maxima of the distance function $d(x,\cdot)$ and $K$ is a maximal clique of $S_k(x)$, we get $K'\setminus K \subseteq S_{k-1}(x)$. On the other hand, by Lemma \ref{l:clique-in-the-sphere}, $K'\setminus K\ne\varnothing$. We claim that $K$ is contained in $\partial X\subseteq S(X)$.  Indeed, let $K'' = K' \setminus \{y\}$ for some $y \in K' \setminus K$ ($y$ exists since $K'\setminus K\ne \varnothing$).
If there was some simplex $K'''$ different from $K'$ and strictly containing $K''$, then there would exist some $t$ adjacent in $G$ to all vertices of $K''$ and not belonging to $K'\setminus K$. By the maximality choice of $K$ in $S_k(x)$, such $t$ would have to belong to the sphere $S_{k-1}(x)$. So $t$ would be adjacent to $y$, and hence to all vertices of $K'$, a contradiction with the maximality choice of $K'$. This contradiction shows that $v\in K\subseteq \partial X\subseteq S(X)$. 
Hence, $v$ would belong to $\partial X$, which is impossible. This proves that $d(v,x)=k-1$ and thus that the entry $d(v,x)$ of $D'$ can be derived from the entry $d(u,x)$ of $D$ and the conditions of Case 2. 
\begin{claim}\label{claim-boundary-S(X)}
    In both Cases 1 and 2, $\partial X' \subseteq S(X')$.
\end{claim}
\begin{proof}
By induction hypothesis, $\partial X\subseteq S(X)$. By definition of $S(X')$, $S(X')=S(X)\setminus \{ u\}$ in Case 1 and $S(X')=S(X)\setminus \{ u\}\cup \{ v\}$ in Case 2. Suppose that there exists $x \in \partial X' \setminus S(X')$. Then there is a non-maximal simplex $K$ of $X'$ containing $x$ that is a facet of a unique simplex $K'$ of $X'$. Then necessarily $K'$ is a maximal clique of $G'$. 
On the other hand, since $x\not\in S(X')$ and $x\in X'$, from the relationships between $S(X')$ and $S(X)$ we get $x \not\in S(X)$. Since $\partial X\subseteq S(X)$, $x\notin \partial X$. 
So the simplex $K$ is contained in a simplex $K''$ of $X$ different from $K'$. Since the simplex $K''$ is not in $X'$, $K''$ has to contain the vertex $u$. Consequently, $x\sim u$. Since $x\notin \partial X$, 
by Lemma \ref{lem:dominated_neighborhood}, $x=v$ and $v\in S(X')\setminus S(X)$. But this contradicts the choice of $x\in \partial X'\setminus S(X')$. This establishes the inclusion $\partial X'\subseteq S(X').$
\end{proof}

From the analysis of Cases 1 and 2 and Claim \ref{claim-boundary-S(X)} we deduce that $\partial X'\subseteq S(X')$ and that $D'$ is the distance matrix of $S(X')$. By the inductive hypothesis, if $Y'$ is a systolic complex such that $\partial Y'\subseteq S(X')\subseteq Y'$ and the distance matrix of $S(X')$ computed in the bridged graph $G(Y')$ coincides with $D'$, then $Y'$ is isometric to $X'$ (under an isometry fixing $S(X')$). Let $G'=G(X')$. From  $G'$ and the distance matrix $D$ of $S(X)$ we reconstruct $G$ by adding the vertex $u$ and all edges $ux$ for all vertices of $G'$ such that $d(v,x)=1$ in $D$. Then $X$ is the clique complex of $G$. Alternatively, $X$ is obtained from $X'$ by adding all cliques of $G$ containing the added vertex $u$. 
Since $u$ is dominated by $v$, by Lemma \ref{lem:dominated_neighborhood} we have that $N[u]\subseteq \partial X \cup \{v\}\subseteq S(X')$, thus all neighbors of $u$ are correctly retrieved. Therefore we conclude that $X$ is uniquely determined by $D$. This concludes the proof of Theorem \ref{thm:boundary_rigidity_systolic}.

The proof of Theorem \ref{thm:boundary_rigidity_systolic} generalizes to weakly systolic complexes.
The only change in the proof is using Theorem \ref{dismantling-bridged}(2) instead of \ref{dismantling-bridged}(1).

\begin{corollary}\label{cor:weakly}
    Let $X$ be a finite weakly systolic complex. Then $X$ is determined up to isomorphism by the combinatorial distances of points of $\partial X$ in $X$. Hence, the class of finite weakly systolic complexes is boundary rigid. 
\end{corollary}

The proof of Theorem \ref{thm:boundary_rigidity_systolic} does not
provide a straightforward algorithm to reconstruct $X$, as one needs
to start by picking a dominated vertex.  However, from our proof it
follows that a systolic complex $X$ has at most
$\sum_{u,u'\in \partial X} d(u,u')$ vertices. Indeed, at each
reconstruction step the distance sum $\sum_{u,u'\in S(X)} d(u,u')$
decreases by at least 1. Consequently, we know that the number of
vertices of $X$ is upper bounded by the sum of the entries of the
initial matrix $D$. We can thus reconstruct all possible systolic
complexes of at most that size.  Then by the existence and uniqueness,
among these complexes, there is one (and only one) complex that has
the same boundary and the same boundary distance matrix as $X$, and we
can return it. This is of course not an efficient algorithm, but it
shows that finite systolic complexes can be reconstructed
algorithmically from their boundary and boundary distance matrix.
In Section \ref{sec:algorithm} we show how to efficiently find a
dominated vertex at each step when the systolic complex is
2-dimensional.  However we do not know how to generalize this approach
for systolic complexes of higher dimension.

\section{Boundary rigidity of Helly complexes} 

The goal of this section is to prove
Theorem~\ref{thm:boundary_rigidity_helly}. Our proof is constructive
and we derive a polynomial time algorithm from it and analyze its
complexity.

\subsection{Proof of Theorem~\ref{thm:boundary_rigidity_helly}}

The proof for Helly complexes goes along the same lines as the proof
of Theorem \ref{thm:boundary_rigidity_systolic}, but we use Lemma
\ref{lem:dominated_far_away_Helly} instead of Lemma
\ref{lem:dominated_far_away} and Lemma \ref{l:Helly-BP} instead of
Lemma \ref{l:clique-in-the-sphere}. Let $X$ be a finite Helly complex
and let $G=(V,E)$ be a Helly graph, which is the 1-skeleton of $X$. As
with systolic complexes, let $S(X)$ be any subset of vertices
containing $\partial X$ and let $D$ be the distance matrix of $S(X)$
in $G$. We proceed by induction on the number of vertices of $X$ and
prove the same inductive hypothesis as for systolic complexes. The
base case where $X$ contains a single vertex, is also the same.

Let $z$ be an arbitrary fixed vertex of $\partial X$ and suppose that $X$ contains at least two vertices. Applying Lemma \ref{lem:dominated_far_away_Helly} to $z\in \partial X$, we conclude that any furthest from $z$ vertex $u$ is a dominated vertex different from $z$. By Lemma \ref{lem:dominated_neighborhood}, $u\in \partial X$.  Suppose that $u$ is dominated in $G$ by a vertex $v$. Let $X'$ be the simplicial complex consisting of all simplices 
of $X$ not containing the vertex $u$ and let $G'$ be the subgraph of $G$ induced by $V\setminus \{ u\}$. Then $G'$ is the 1-skeleton of $X'$. Since $u$ is dominated by $v$, $G'$ is an isometric subgraph of $G$. By the second assertion of Lemma \ref{lem:dominated_far_away_Helly},  $G'$ is Helly. Thus $X'$ is the clique complex of  $G'$, whence $X'$ is Helly. 

Next, from $S(X)$ and the distance matrix $D$ of $S(X)$ in $G$, we define a set $S(X')$ such that $\partial X'\subseteq S(X')\subseteq V'=V\setminus \{ u\}$ and we show how to compute the distance matrix of $S(X')$ in $G$ (and thus in $G'$). We distinguish the same two cases. The analysis of Case 1 is completely the same because it uses only Lemma \ref{lem:dominated_neighborhood} which holds for all dismantlable graphs. 
Now we consider Case 2, i.e., the case where there does not exist a vertex $v'\in S(X)$, $v'\sim u$ such that $N[u]\cap S(X)\subseteq N[v']$. 
In this case, we also set $X' = X\setminus\{u\}$, and $S(X') = S(X)\setminus\{u\}\cup\{v\}$. Pick any $x\in S(X')$ and we want to derive 
the distance $d(v,x)$.  If $x\in N(u)$, then $d(v,x) = 1$ since $v$ dominates $u$. Otherwise, let $k=d(x,u) \geq 2$. If there exists $w \in S(X')\cap N(u)$ with $d(x,w) = k+1$, then as in case of systolic complexes we can argue that $d(x,v) = k$. Now suppose that $d(w,x)\le k$ for every vertex $w\in S(X')\cap N(u)$. We assert that in this case $d(v,x)=k-1$. 
Applying Lemma \ref{l:Helly-BP} to $x$ and $u$, there exists a 
vertex $v'\sim u$ at distance $k-1$ from $x$ and adjacent to all 
vertices $w\in N[u]\cap B_k(x)$. Since any vertex  $w\in N[u]$ has distance at most $k$ from $x$, we conclude that $v'$ dominates $u$ 
in $G$. Since $v$ also dominates $u$ and $v\notin S(X)$, 
necessarily $v'=v$, yielding $d(v,x)=k-1$. This allows to 
compute the distance matrix $D'$ of $S(X')$. Finally, the proof 
of Claim \ref{claim-boundary-S(X)} and the final application of the induction are the same as for systolic complexes. This establishes 
Theorem \ref{thm:boundary_rigidity_helly}. 

\subsection{The complexity of the reconstruction}

In contrast to systolic complexes, at each step of the reconstruction, a dominated vertex $u$ of a Helly complex $X$ can be easily computed. Thus, the proof of Theorem \ref{thm:boundary_rigidity_helly} gives an algorithm to reconstruct $X$ in polynomial time in the number of its
vertices. We now give an upper bound on its complexity.

During the execution of the algorithm, we have the subgraph $G'$ of
$G$ induced by all vertices discovered so far (at the end, $G'$
coincides with the 1-skeleton $G$ of $X$), a set of vertices $S(X)$
containing the vertices of the boundary of the current complex (and potentially other vertices), the
distance matrix $D$ of $S(X)$ (all distances are computed in $G$), and
the subgraph $G[S(X)]$ of $G$ induced by $S(X)$. Initially, we are
given $S(X) = \partial X$ and the distance matrix $D$ of $\partial
X$. Let $n_0=|\partial X|, n=|V(X)|,$ and $m=|E(X)|$ (only $n_0$ is
initially known to the algorithm).  From $D$, we can compute in time
$O(n_0^2)$ the graphs $G' = G[S(X)]$ by taking all pairs $x,y\in S(X)$
with $d(x,y)=1$ as edges.  At each iteration, we update $S(X)$, $D$,
$G[S(X)]$, and $G'$.

At each step we have $|S(X')| = |S(X)|-1$ or $|S(X')| =
|S(X)|$, hence $|S(X)| \leq n_0$ at all iterations. Initially we sort the vertices of $\partial X$ into a list $L$ ordered by
decreasing distances to $z$; this can be done in $O(n_0)$ time using bucket sort. At
each step, we can then find a vertex $u$ at maximum distance from $z$
in constant time. To check whether $u$ is dominated in $G[S(X)]$, for
each neighbor $v$ of $u$, we check if $|N(u) \cap N(v)| = |N(u)|-1$;
this can be done in $O(|N(u)| +|N(v)|)$ time assuming that there is a
global order on the vertices and that the adjacency list of each vertex
is initially sorted according to this order. We are considering the
neighborhood in $G[S(X)]$, and thus
$|N(u)|+|N(v)| = O(|S(X)|) = O(n_0)$. Consequently, one can
check in $O(n_0\cdot \deg(u))$ time if
$S(X') = S(X) \setminus \{u\}$ or if
$S(X') = S(X) \cup \{v\} \setminus \{u\}$. Note that we can insert $v$
in $L$ in constant time if we keep for each $0 \leq k \leq n_0$ a
pointer to the first element $w$ of the list $L$ such that
$d(z,w) = k$. We can thus assume that at all steps, $L$ contains the
list of vertices of $S(X)$ ordered by their distance to $z$.

At each step where $S(X') = S(X) \setminus \{u\} \cup \{v\}$, we
should compute $d(x,v)$ for any vertex $x \in S(X) \setminus
\{u\}$. We know that either $d(x,v) = d(x,u)$ or $d(x,v)=d(x,u)-1$,
and that we are in the second case if and only if
$d(x,u) = \max \{d(x,w): w \in S(X) \text{ and } w \sim
u\}$. Consequently, for every vertex $w$ in $S(X)$, we can compute
$d(x,v)$ in $O(\deg(u))$ time and we can thus update the distance
matrix $D$ as well as the graphs $G'$ and $G[S(X)]$ in
$O(n_0\cdot \deg(u))$ time. Therefore, for any Helly complex $X$,
starting from the boundary $\partial X$ of size $n_0$ and its distance
matrix, one can reconstruct $X$ in $O(m\cdot n_0)$ time.

\section{Reconstruction of 2-dimensional systolic complexes}\label{sec:algorithm}

In this section we prove Theorem \ref{thm:boundary_rigidity_systolic-2dim} for 2-dimensional systolic complexes. The key goal is to algorithmically find a dominated vertex in $X$. Analogously to Helly graphs, every dominated vertex of a bridged graph $G=G(X)$ belongs to the boundary of $X$. But unlike Helly graphs, not every furthest from $z$ vertex of $G$ is a dominated vertex. In case of $K_4$-free bridged graphs (a.k.a. 2-dimensional systolic complexes) we find a dominated vertex $u$ in two stages. In the proofs, we will use a consequence of Lemma \ref{l:clique-in-the-sphere} that a sphere $S_k(z)$ of a 2-dimensional systolic complex $X$ cannot contain a triangle; in other words, $X$ cannot contain a triangle \emph{equidistant} to some vertex $z$. 

Note that there exist 2-dimensional systolic complexes with $n$ vertices and $O(n^{\frac{4}{3}})$ edges. For this, consider the construction of Benson \cite{Be} and Singleton \cite{Si} of a regular bipartite graph $H$ with girth 8, degree $q$, and $2(q^3 + q^2 + q+1)$ vertices. Let $G$ be a graph obtained by adding to $H$ a new vertex adjacent to all vertices of $H$. Then $G$ is bridged, $K_4$-free, and will contain 
$n:=2(q^3 + q^2 + q+1)+1$ vertices and $2(q^4 + 2q^3 +2q^2+2q+1)=O(n^{\frac{4}{3}})$ edges. Other examples of 2-dimensional systolic complexes with a non-linear number of edges can be obtained by considering finite simply connected subcomplexes of the Weetman's graph $W(H)$, where $H$ is the Benson-Singleton graph or other similar graph.  

\subsection{The computation of a dominated vertex}
Let $X$ be a 2-dimensional systolic complex and let $G=G(X)$ be its
1-skeleton. We start with $z_1\in \partial X$ and let
$k_1 = \max_{v \in \partial X}\{d(z_1,v)\}$. By Lemma
\ref{l:sphere-in-boundary}, all vertices of the sphere $S_{k_1}(z_1)$
belong to $\partial X$. Let $S_1$ be a connected component of the
subgraph $G_1$ of $G$ induced by the sphere $S_{k_1}(z_1)$. Since
$S_{k_1}(z_1)\subseteq \partial X$, $G_1$ can be constructed from
$S_{k_1}(z_1)$ and the distance matrix $D$ of $\partial X$. If $S_1$
contains a unique vertex $u$, then $N(u) \subseteq B_{k_1-1}(z_1)$ and
thus, by Lemma~\ref{l:clique-in-the-sphere}, $N(u)$ is a complete graph and $u$ is dominated (by any of its neighbors).

Suppose now that $|S_1| \geq 2$.  Pick any $z_2 \in S_1$ and consider
a vertex $u \in S_1$ such that $u$ has a neighbor $v$ in $S_1$ with
$d(z_2,v) = d(z_2,u)-1$ and $u$ has no neighbor $w$ in $S_1$ with
$d(z_2,w) = d(z_2,u)+1$. Observe that we can find such a vertex $u$
among the vertices at maximum distance from $z_2$ in the connected
component $S_1$.
We want to show that $u$ is a dominated vertex. Let $k_2 = d(z_2,u)$.
For that we need to consider two cases depending on whether $u$ has neighbors in $S_1$ at distance $k_2$ from $z_2$ or not. 

\begin{mycase}
  For every $w \in S_1 \cap N(u)$, $d(z_2,w) = d(z_2,u) -1 = k_2-1$.
\end{mycase}

We claim that in this case, $v$ is the only neighbor of $u$ in $S_1$.
Indeed, suppose that there exists $w \in N(u) \setminus \{v\}$ such
that $d(z_1,w) = k_1$. Since $d(z_2,w) = k_2 -1$, by
Lemma~\ref{l:clique-in-the-sphere}, we get a triangle $uvw$
equidistant from $z_1$, a contradiction.

By the triangle condition (TC) there exists a vertex $x$ with
$x \sim v,u$ and $d(z_1,x) = k_1-1$. We assert  that $x$ dominates
$u$. Indeed, for any $w \sim u$ different from $v$, we know that
$d(z_1,w) = k_1 -1$ and thus $x \sim w$ by
Lemma~\ref{l:clique-in-the-sphere}.

\begin{mycase}\label{mycase-2}
  There exists $w \in N(u) \cap S_1$ such that $d(z_2,w) = d(z_2,u) = k_2$. 
\end{mycase}

For any $w \in N(u) \cap S_1$ such that $d(z_2,w) = k_2$, by (TC),
there is a vertex $x \sim u, w$ such that $d(z_2,x) = k_2-1$. Note
that $d(z_1,x) = k_1 -1$ as otherwise $uxw$ is a triangle equidistant
from $z_1$, which is impossible. This implies in particular that
$x \neq v$.

Observe also that $v$ and $x$ are the unique neighbors of $u$ at distance
$k_2-1$ from $z_2$, as otherwise, by
Lemma~\ref{l:clique-in-the-sphere}, there would be a triangle at
distance $k_2-1$ from $z_2$. Consequently, $x$ is adjacent to all
vertices $w \in N(u) \cap S_1$ such that $d(z_2,w) = k_2$.  Note
also that for any $y \in N(u)$ such that $d(z_1,y) = k_1-1$ or
$d(z_2,y) = k_2-1$, we also have $x \sim y$ by
Lemma~\ref{l:clique-in-the-sphere} with respect to $z_1$ or
$z_2$. Consequently, $x$ is adjacent to all neighbors of $u$ and thus
$u$ is dominated by $x$.

\medskip

We have thus shown how to efficiently compute a dominated vertex. Given such a vertex, we can iteratively apply the process described in the proof of Theorem \ref{thm:boundary_rigidity_systolic} to reconstruct $X$ in polynomial time.  

\subsection{The complexity of the reconstruction}

As in the case of Helly complexes, during the execution of the
algorithm, we have a graph $G'$ containing all vertices discovered so
far, a set of vertices $S(X)$ containing the vertices of the boundary
of the current complex (and potentially other vertices), the subgraph
$G[S(X)]$ of $G$ induced by $S(X)$, and a distance matrix $D$
containing the distances between the vertices of $S(X)$.  At each step
of the reconstruction, once we have found a dominated vertex $u$, we
can perform exactly the same steps in the same time complexity as in
the case of Helly complexes. Consequently, we now focus on the time
complexity of the identification of a dominated vertex at each
step. As before, let $n_0=|\partial X|, n=|V(X)|,$ and $m=|E(X)|$.

We fix a vertex $z_1$ that remains the same during the whole execution
of the algorithm. As in the case of Helly complexes, we maintain a
list containing the vertices of $S(X)$ sorted by decreasing distance
to $z_1$.
Initially, we pick a vertex $z_2$ at maximum distance $k_1$ from $z_1$
and we start executing a Breadth-First-Search traversal from $z_2$
considering only the vertices at distance $k_1$ from $z_1$ (that are
all contained in $S(X)$). For each vertex $u$ considered during this
traversal, we store the number $c(u)$ of neighbors $w$ of $u$ such
that $d(z_1,w) = d(z_1, u) = k_1$ and $d(z_2,w) = d(z_2,u)+1$. Note
that we process each vertex $u$ in $O(\deg(u))$ time.  When we find a
vertex $u$ with $c(u) = 0$, $u$ satisfies the conditions to be the
vertex chosen in the previous subsection and we thus know that $u$ is
a dominated vertex. We can then construct $S(X')$ and update $D$,
$G[S(X)]$ and $G'$ as in the case of Helly complexes. Once we have
done this, for each neighbor $w$ of $u$ such that
$d(z_1,w) = d(z_1, u) = k_1$ and $d(z_2,w) = d(z_2,u)-1$, we decrease
$c(w)$ by $1$. Note that such a vertex $w$ has been considered before
$u$ in the Breadth-First-Search traversal of $G$ and that once $u$ has
been removed from $S(X)$, $w$ has one less neighbor at distance $k_1$
from $z_1$ and further away from $z_2$. Moreover, note that in the
case where $S(X') = S(X) \setminus \{u\} \cup \{v\}$, the vertex $v$
is at distance $k_1 - 1 = d(z_1,u) -1$ from $z_1$ and thus adding $v$
to $S(X)$ does not impact the value of $c(w)$ for any $w$ at distance
$k_1$ from $z_1$.  As long as there exists a vertex $u$ with
$c(u) = 0$, we reiterate this process. If there is no such vertex $u$,
we continue the Breadth-First-Search traversal rooted at $z_2$ from
where we stopped before. If the vertex $u$ that has been removed is
equal to $z_2$, then we pick as $z_2$ another vertex $z'_2$ at maximum
distance from $z_1$ in $S(X)$. The algorithm ends once we have
processed $z_1$.

With this algorithm, each vertex $u$ is considered in a
Breadth-First-Search traversal only once all along the execution. This
takes only $O(\deg(u))$ time for each vertex $u$. Consequently, the
total cost to identify a dismantlable vertex at each step of the
execution is $O(\sum_{u \in V(X)} \deg(u)) = O(m)$ overall. Therefore,
the complexity of our reconstruction algorithm is dominated by the
cost of the updates of $S(X)$, $D$, $G[S(X)]$, and $G'$: for any
2-dimensional systolic complex $X$ with $m$ edges, starting from its
combinatorial boundary $\partial X$ of size $n_0$, we can reconstruct
$X$ in $O(n_0 \cdot m)$ time.

\section{Final remarks}

In this article, we proved that finite systolic and Helly complexes  
are boundary rigid.  This generalizes the results of~\cite{Ha}, provides a partial answer to Question 2 of~\cite{HaScTaTa}, and  confirms a question 
of~\cite{ChCh_bd}. From our proofs, it follows that a systolic or Helly complex 
$X$ has at most $\sum_{u,u'\in \partial X} d(u,u')$ vertices. 
We also showed that finite Helly complexes  and 2-dimensional systolic complexes can be reconstructed in 
time polynomial in their number of vertices.  However, we did 
not succeed to prove a similar result for all systolic 
complexes. The bottleneck of our approach is how to efficiently compute a dominated vertex, based only on the
boundary information. We leave this as an open question. 

\begin{question}
    Given the distance matrix $D$ of the boundary $\partial X$ of a finite systolic complex $X$, is there an algorithm polynomial in the number of vertices of $X$ for finding a dominated vertex of $X$, and thus for reconstructing $X$?
\end{question}

Our algorithms for the reconstruction of Helly complexes and 2-dimensional systolic complexes run in $O(n_0\cdot m)$ time. Our reconstruction method seems to not be implementable with better complexity bounds. Nevertheless, it will be interesting to investigate if the reconstruction can be done faster: 

\begin{question}
Is it possible to reconstruct a finite 2-dimensional systolic
complex (respectively, a Helly complex) from its boundary in 
quadratic $O(n^2)$ time (or in a time linear in the size of the
input and output, i.e., $O(n_0^2+m)$)?  
\end{question}

Question 2 of Haslegrave et al.~\cite{HaScTaTa} asks under what conditions is a finite simplicial complex reconstructible from its boundary distances. In our paper, the dismantlability of bridged and Helly graphs was essential. Therefore, one can ask whether the class of clique complexes of dismantlable graphs is boundary rigid.  This is not true, as the following simple example shows.

\begin{figure}[h]
\begin{center}
\includegraphics{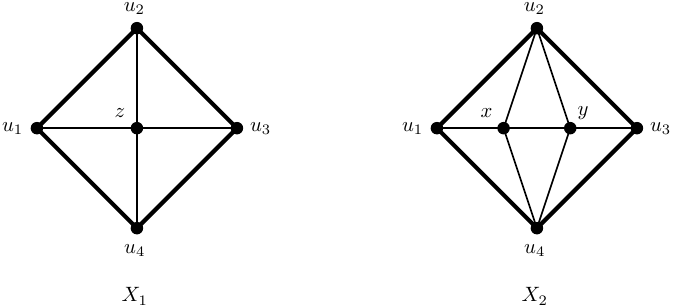}
\end{center}
\caption{Two simplicial complexes $X_1$ and $X_2$ with dismantlable 1-skeletons. 
$X_1$ and $X_2$ cannot be distinguished from their
  boundary information (the boundary of each complex is the downward
  closure of the thick edges).}\label{fig-cex-dism}
\end{figure}

\begin{example} \label{ex:dism_not_reconstr} Let $X_1$ be the clique
  complex of the graph $G_1$, which is the 4-wheel (a square
  $S=u_1u_2u_3u_4$ plus a vertex $z$ adjacent to all vertices of $S$).
  Let $X_2$ be the clique complex of the graph $G_2$ which consists of
  the square $S$ plus two adjacent vertices $x$ and $y$, where both
  $x$ and $y$ are adjacent to $u_2,u_4$ and, additionally, $x$ is
  adjacent to $u_1$ and $y$ is adjacent to $u_3$ (See
  Figure~\ref{fig-cex-dism}). Then one can easily see that
  $\partial X_1$ and $\partial X_2$ coincide with the square $S$, and
  $S$ isometrically embeds in $G_1$ and $G_2$ (namely, in both graphs
  $d(u_1,u_3)=d(u_2,u_4)=2$ and
  $d(u_1,u_2)=d(u_2,u_3)=d(u_3,u_4)=d(u_4,u_1)=1$).  Both $G_1$ and
  $G_2$ are dismantlable: $u_1, u_2, u_3, u_4, z$ is a dismantling
  order of $G_1$ and $u_1, u_2, u_3, u_4, x, y$ is a dismantling
  order of $G_2$. 
\end{example} 

However, it will be interesting to investigate which classes of dismantlable complexes are boundary rigid. 
One  candidate class is the class of the clique complexes of the square $G^2$ of graphs $G$ with convex balls ($G^2$ has the same vertex-set as $G$ and $u,v$ are adjacent in $G^2$ if $d(u,v)\in \{ 1,2\}$ in $G$). It was shown in \cite{ChChGi} that for any graph $G$ with convex balls, any BFS order of $G$ is a dismantling order of $G^2$. While graphs with convex balls generalize bridged and weakly bridged graphs, their squares are not bridged or weakly bridged \cite{ChChGi}. 

\begin{question} Is the class of the clique complexes of the squares $G^2$ of graphs with convex balls $G$ boundary rigid? 
\end{question}

A related question is whether our Theorems \ref{thm:boundary_rigidity_systolic} and \ref{thm:boundary_rigidity_helly} admit a common generalization:

\begin{question} Is there a characterization of the class of dismantlable (or  weakly modular) graphs whose clique complexes are boundary rigid? 
\end{question}

A positive answer to this question seems nontrivial to us for two reasons. First, although in both proofs we use dismantlability, in each case the distance matrix is updated  using  different metric properties of the respective classes of graphs. Second, finding such a common generalization would imply the existence of a metric geometry generalizing systolic and injective (Helly) geometries, which seem to be quite apart. However, if it exists, then it will offer the possibility of a unified treatment of systolic and Helly groups. 

\medskip
It will be also interesting to investigate the discrete variants Conjecture 1.4 of Burago and Ivanov \cite{BuIv} about simple manifolds as minimal fillings. We say that a class $\mathfrak{X}$ of simplicial (or cube) complexes satisfies the \emph{minimal filling} property if for any two complexes $X$ and $Y$ from $\mathfrak{X}$ such that $X$ and $Y$ have isomorphic boundaries and $d_{G(X)}(x,y)\le d_{G(Y)}(x,y)$ for any $x,y\in \partial X\cong \partial Y$, the inequality $|X^0|\le |Y^0|$ holds.
Namely, we replace the volume of the manifold as in \cite{BuIv} by the number of vertices. More generally, we say that a class  $\mathfrak{X}$ of simplicial (or cube) complexes satisfies the \emph{universal minimal filling} property if for any  complex $X\in {\mathfrak X}$ and any simplicial (cube) complex $Y$ such that $X$ and $Y$ have isomorphic boundaries and $d_{G(X)}(x,y)\le d_{G(Y)}(x,y)$ for any $x,y\in \partial X\cong \partial Y$, the inequality $|X^0|\le |Y^0|$ holds. We conclude the paper with the following question: 

\begin{question} Do the classes of systolic, Helly, or CAT(0) cube complexes satisfy the minimal filling property? Do some of these classes satisfy the universal minimal filling property?  
\end{question}

\end{document}